\documentclass{amsart}

\usepackage{amssymb}

\usepackage{mathtools}

\newtheorem{theorem}{Theorem}[section]
\newtheorem{lemma}[theorem]{Lemma}

\theoremstyle{definition}
\newtheorem{definition}[theorem]{Definition}

\theoremstyle{proposition}
\newtheorem{proposition}[theorem]{Proposition}

\theoremstyle{corollary}
\newtheorem{corollary}[theorem]{Corollary}

\numberwithin{equation}{section}

\def\fg{\mathfrak{g}}

\begin{document}

% \title[short text for running head]{full title}
\title[Binomial Transforms]%
{Binomial Arrays and \\ Generalized Vandermonde Identities} 

%    Only \author and \address are required; other information is
%    optional.  Remove any unused author tags.

%    author one information
% \author[short version for running head]{name for top of paper}
\author{Robert W. Donley, Jr.}
\address{Department of Mathematics and Computer Science,  Queensborough Community College (CUNY), Bayside, NY 11364, USA}
\curraddr{}
\email{RDonley@qcc.cuny.edu}

%    author two information

%    The 2010 edition of the Mathematics Subject Classification is
%    now available.  If you are citing a classification from the
%    new scheme, use the following input coding instead.

\begin{abstract}
In previous work on Clebsch-Gordan coefficients, certain remarkable hexagonal arrays of integers are constructed that display behaviors found in Pascal's Triangle.  We explain these behaviors further using the binomial transform and discrete convolution. Here we begin by introducing the notion of a binomial array and develop several ``hockey stick" rules.
Then we give an algorithm that generalizes the classical Vandermonde Identity; this produces infinite families of summation formulas, which we use to expand and prove certain combinatorial identities for  the Catalan numbers.  Finally, we recast the theory in terms of the finite-dimensional representation theory of $SL(2, F)$.
\end{abstract}

\maketitle

%    Text of article.

% section 1
\section{Introduction}

Recent work of the author on Clebsch-Gordan coefficients allows for wide application of the methods of combinatorial analysis, driven in particular by generating functions, recurrences, finite group actions, and Pascal's triangle.  The basic combinatorial rules established in \cite{Do} receive a somewhat thorough application in \cite{D2}, save for the orthogonality relations. The goal of this work is to draw further analogies from Pascal's triangle, in particular characterizing the first orthogonality relation as a special case of discrete convolution.  

Coupling this observation with methods of Dwyer (\cite{DM}, \cite{Dw}) and  Frankel (\cite{EF}),  we also explain the invariance property of orthogonality relations under uniform inward or outward shifting of columns. This method yields vast extensions of the summation formulas associated to Pascal's recurrence.

Key to understanding these phenomena is the notion of  a binomial array;  this concept extends both Pascal's triangle and forward difference tables and has properties in accordance with  Riordan arrays \cite{Ba}.  In turn, the binomial array is a natural implementation of the generalized binomial transform.  Instead of producing a new sequence from a given sequence, a consideration of all possible binomial transforms and their inverses produce an array of values.  The binomial transform here may be regarded as iteration of Pascal's recurrence and its inverse as iteration of alternating partial sums.

Examples of these arrays are both old and well-represented in the literature, including
\begin{enumerate}
\item the extended Pascal triangle (Figure 1; compare with Table 1 in \cite{EF}), binomial coefficients, and cumulative (or figurate) numbers,
\item difference tables and the calculus of finite differences (for instance, \cite{GN}, Ch. 5.3),
\item moment and correlation calculations in statistics (\cite{DM}, \cite{Dw}),
\item Catalan numbers (A000108) and Catalan triangles (A009766 - in order of increasing detail, see \cite{Sh}, \cite{Fo}, \cite{KM}, and Figure 5 below; see also Figure 3.2 in \cite{GL}),
\item Bell numbers (A000110) and Bell's triangle (A011971),
\item Motzkin numbers (A001006) as finite differences of Catalan numbers (\cite{St2}, p. 126.),
\item rencontres numbers (A008290) and derangement triangles (\cite{Sp}),
\item Clark's triangles (A090850), and
\item hexagons of Clebsch-Gordan coefficients (\cite{Do}, \cite{D2}; see also \cite{Vi}, III.8.7, p. 188, for finite differences).
\end{enumerate}
Here entries of the OEIS (\cite{Sl}) are referenced as usual with digits following a leading A.

\begin{table}
\caption{Shapiro's Catalan Triangle for $B_{n, k}$}
\label{tab:1}
\begin{tabular}{|p{1cm}|p{1cm}|p{1cm}|p{1cm}|p{1cm}|p{1cm}|p{1cm}|}
\hline\noalign{}
 $n\backslash k$ & 1 & 2 & 3 & 4 & 5 & 6\\
 \noalign{}\hline\noalign{}
1 & 1 &  & & & &   \\
 2 & 2 & 1 & & & &   \\
 3 & 5 & 4 & 1 & & &   \\
 4 & 14 & 14 & 6 & 1 & &   \\
 5 & 42 & 48 & 27 & 8  & 1 &   \\
 6 & 132 & 165 & 110 & 44 & 10 & 1   \\
\noalign{}\hline\noalign{}
\end{tabular}
\end{table}

In addition to providing another setting for Clebsch-Gordan coefficients, these methods allow an extension of summation formulas due to Shapiro.  In particular, if we define \begin{equation*}B_{n, k}=\frac{k}{n}\begin{pmatrix}2n \\ n-k\end{pmatrix},\end{equation*} then an extension of the Catalan numbers is  given by Table 1.  These numbers satisfy the equation
\begin{equation*}B_{n, k}=B_{n-1, k-1}+ 2B_{n-1, k}=B_{n-1, k+1},\end{equation*}
which may be interpreted as the square of Pascal's recurrence.   This triangle occurs as a fundamental example of a Riordan array \cite{Ba}, and it is known (\cite{Sh}; \cite{Ko}, Theorem 13.1) that the length squared of any row is a Catalan number, as is the dot product of any two rows.  Extensions of these two results (Corollaries 8.4 and 8.5) and placement into a wider context, adjunct to the Riordan array model,  are two concerns of  the present work.

Other new features include
\begin{enumerate}
\item definitions of binomial array and  the general binomial transform (section 2), 
\item symmetries of binomial arrays (section 3),
\item nine hockey stick rules - six long and three short (section 4),
\item six Vandermonde identities for convolution (sections 5),
\item examples and two families generalizing the sequence of Catalan numbers (sections 6 through 8), and
\item an identification of the model in terms of the representation theory for $SL(2, F)$  (section 9).
\end{enumerate}

For notation, we remind the reader of the usual conventions for binomial coefficients, which may be invoked without comment.  For any integers $k$ and $n$, 
 \begin{equation}\begin{pmatrix} n \\ k\end{pmatrix}= \begin{cases}{\frac{n!}{k!(n-k)!}} & 0\le k\le n,\\
 \frac{(-1)^k(-n+k-1)!}{k!(-n-1)!} & n<0,\ k\ge 0,\  and\\
 0 & otherwise.\end{cases}\end{equation}
 See Figure 1, with columns indexed by $n$ and rows indexed by $k\ge 0.$ While most of our examples  require only the integers, we will  assume constants lie in a field $F$ of characteristic zero.
 
 All figures in this work were created using Excel.

%section 2 
\section{Binomial Arrays}
Let $V=F[[x]]$ be the algebra of power series in one variable with coefficients in $F$, and let $\mathcal{B}=\{1, x, x^2, \dots\}$ be a basis for the subalgebra of polynomials in  $V.$ Every element in $V$ can be uniquely represented as a formal infinite sum of elements in $\mathcal{B}.$  Using this representation, $V$ is isomorphic to the ring of sequences $\{a_i\}_{i=0}^\infty.$

\begin{definition}
Suppose $\{a_i\}_{i=0}^\infty$ is a sequence. We define the {\bf binomial array}  $B(a_i)$ to be the lower half-plane array defined as follows:
\begin{enumerate}
\item All entries in the top line equal $a_0.$ We index  the top line as row 0 and all rows below in increasing order,
\item The entry in the zeroth column and $i$-th row is $a_i$, and
\item All other entries are defined using Pascal's recurrence;  that is, if the entry in row $k$ and column $n$ is denoted by $a_{k, n}$ then 
\begin{equation*}a_{k, n+1}= a_{k-1, n}+a_{k, n}.\end{equation*}
\end{enumerate}
If $p(x)=a_0+a_1x+\dots+a_mx^m$ is a polynomial of degree $m$, we define $B(p(x))$ to be the binomial array  $B(a_i)$ with $a_i=0$ for all $i>m.$  For general $a_i$ with generating function $p(x),$ we may likewise use $B(p(x))$ to denote $B(a_i)$ when convenient.
\end{definition}
Of course, $a_{i, 0}=a_i.$  Pascal's recurrence allows us to extend the middle column to both the left and right sides of $B(p(x))$.  The recurrence occurs as a capital L pattern; the vertical entries sum to define the toe of the L.  Knowing any two entries on the L allows calculation of the third.

\begin{figure}
\includegraphics[scale=.6]{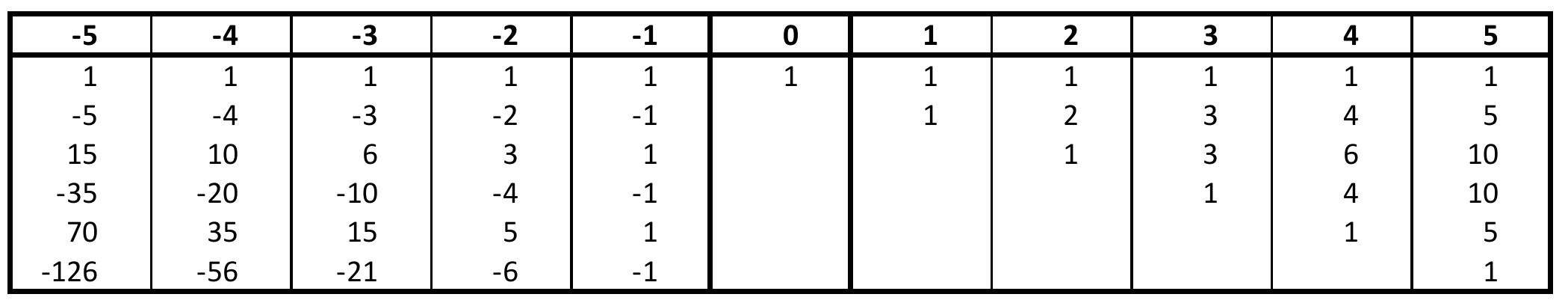}
\caption{Extended Pascal's Triangle}
\label{}
\end{figure}

\begin{definition}
For $n\ge 0,$ we define the  {\bf (extended) binomial transform} $B^na_{k}$ of $a_i$ to be the sum
\begin{equation}B^na_{k}=a_{k, n}=\sum\limits_{i=0}^n \begin{pmatrix} n\\ i\end{pmatrix}a_{k-i}.\end{equation}
\end{definition}
Several definitions of binomial transform appear in the literature; see, for instance, \cite{Boy1} or \cite{Sp}. One version is given as $B^na_n$; our definition extends this binomial transform by collecting all such transforms for all consecutive subsequences of $a_i.$ 
\begin{definition}
For $n>0,$ the {\bf inverse binomial transform} $B^{-n}a_{k}$ of $a_i$ is defined by the sum
\begin{equation}B^{-n}a_{k}=a_{k, -n}=\sum\limits_{i=0}^n (-1)^i  \begin{pmatrix} n+i-1\\ i\end{pmatrix}a_{k-i}.
\end{equation} 
\end{definition}
If $p(x)=\sum\limits_{i=0}^\infty a_ix^i,$ then $B^na_{k}$ is the coefficient of $x^k$ in $(1+x)^np(x),$ and the definition follows from the Binomial Theorem.  Thus if each column of a binomial array represents coordinates for a  power series, Pascal's recurrence is implemented by multiplying $p(x)$ by $(1+x)$ repeatedly, and this process is inverted with division by $(1+x)$.  For instance, if
\begin{equation*}B^1p(x)=(1+x)p(x)=\sum\limits_{i=0}^\infty b_ix^i,\end{equation*}
then
\begin{equation*}b_k=a_k+a_{k-1}.\end{equation*}
On the other hand, inversion is given by multiplication of power series; if
\begin{align*}B^{-1}p(x)&=\frac{p(x)}{1+x}=(1-x+x^2-x^3+\dots)p(x)\\
&=\sum\limits_{i=0}^\infty b_ix^i.\end{align*}
then
\begin{equation*}b_k=a_k-a_{k-1}+\dots+(-1)^ka_0.\end{equation*}
The formula for the inverse binomial transform then follows from the binomial series expansion:  if $n\ge 0,$ then
\begin{equation*}(1+x)^{-n}=\sum\limits_{i=0}^\infty (-1)^i \begin{pmatrix} n+i-1\\ i\end{pmatrix}x^i \end{equation*}
\begin{proposition}[Hockey Stick Rule]   Let $B(a_i)$ be a binomial array.  If column $n$ is given, then column $n-1$ is given by alternating partial sums of  column $n$.  In particular,  
\begin{equation*}a_{k, n-1} = a_{k, n}-a_{k-1, n}+a_{k-2, n} +\dots +(-1)^ka_{0, n}.\end{equation*} 
\end{proposition}
In practice, it may be inconvenient to sum to the top line.  The sum may be shortened if a higher entry in column $n-1$ is known.  That is, one may compute an alternating partial sum by continuing a known alternating partial sum of the same entries.
\begin{proposition}[Short Hockey Stick Rule]  Suppose $0 < k_1 < k_2.$
\begin{equation*}a_{k_2, n-1} = (a_{k_2, n}-a_{k_2-1, n} +\dots +(-1)^{k_2-k_1}a_{k_2+1, n})+(-1)^{k_2-k_1+1}a_{k_2, n-1}.\end{equation*} 
\end{proposition}

The following subarrays indicate the pattern of each rule:
\begin{equation*}  
\begin{bmatrix*}[r]
 1 & 1  & 1   & 1   & \fbox{1} & 1  \\
 6 & 7 & 8 &   9  & \fbox{10}   & 11  \\
 15 & 21 &  28  & 36 & \fbox{45} & 55 \\
 20&  35   & 56   & 84  & \fbox{120}  & 165  \\
 29 &  49 &  84 & \fbox{140} & \fbox{224} & 344\\
\end{bmatrix*},\quad
\begin{bmatrix*}[r]
 1 & 1  & 1   & 1   & 1 & 1  \\
 6 & 7 & 8 &   \fbox{9}  & 10   & 11  \\
 15 & 21 &  28  & 36 & \fbox{45} & 55 \\
  20 & 35   &  56   & 84   & \fbox{120}  & 165  \\
  29 &  49 &  84 & \fbox{140} & \fbox{224} & 344\\
\end{bmatrix*}.
\end{equation*}
Of course,
\begin{equation*}140=224-120+45-10+1,\qquad 140=(224-120+45)-9.\end{equation*}

It will often be convenient to change the origin of a binomial array.  If  column 0 is  considered as an initial condition at time $t=0$, then the rule  
\begin{equation*}B^{t_1}B^{t_2}a_i=B^{t_1+t_2}a_i\end{equation*}
holds; for the system to evolve beyond time $t$, only the values at time $t$ are needed.  

As a consequence, we may shift the entire array to the left or right when convenient, which amounts to a uniform shift in the second index of $a_{k, n}.$  Uniform vertical shifting may  be performed by multiplying the polynomials for each column by $x^k$; when $k$ is negative,  columns must then allow for  values in the ring of Laurent power series.  In this work, we require that all values above row 0 equal zero.
\vspace{10pt} 
 
We now  extend the vector space operations on powers series and sequences to binomial arrays.
\begin{definition}
Let $\{a_i\}_{i=0}^\infty$ and $\{b_i\}_{i=0}^\infty$ be sequences.  If the sum and scalar multiplication operations are defined as usual for matrices, then 
we have the sum 
\begin{equation*}B(a_i)+B(b_i)= B(a_i+b_i)\end{equation*}
and the scalar multiple 
\begin{equation*}rB(a_i)=B(ra_i).\end{equation*}
\end{definition}
In effect, the binomial transform is linear, so binomial arrays respect linear operations. For instance, 
\begin{proposition}  Any finite linear combination of binomial arrays is a binomial array.
\end{proposition}

With this proposition in mind, each binomial array may be expressed in terms with the extended Pascal's triangle as the atomic element.  Now let $p(x)=1$; that is, define 
\begin{equation}e_i=\begin{cases} 1 & i=0\\ 0 &i>0\end{cases},\quad \text{so that}\quad e_{k, n}= \begin{cases}\begin{pmatrix} n \\ k\end{pmatrix} & n\ge 0\\
 (-1)^k\begin{pmatrix} -n+k-1 \\ k\end{pmatrix} & n<0.\end{cases}\end{equation}
Then $P(0)=B(p(x))$ is just the binomial array with Pascal's triangle on each side;  on the right, the rows of the usual Pascal's triangle are oriented as columns from the top line, and, on the left, the rows are oriented as diagonals from the top line to column $-1$, now with signs alternating based on row number parity. See Figure 1 above.

\begin{definition}For $j\ge 0$, define the sequence $\{e^j_i\}_{i=0}^\infty$ by 
\begin{equation}e^j_i=\begin{cases} 1 & i=j,\\ 0 & otherwise\end{cases};\end{equation}
that is, $e^j_i$ is the zero sequence with a single 1 in the $j$-th position.
\end{definition}
\begin{definition}  Fix $j\ge 0.$  Define $P(j)=B(e^j_i).$  That is, $P(j)$ is the binomial array corresponding to Pascal's triangle, but with the top row of ones moved to row $k$ and zeros in all rows above.
\end{definition}
It follows that every $B(a_i)$ is a possibly infinite linear combination of binomial arrays of type $P(k).$  Specifically,
\begin{equation*}B(a_i) = \sum\limits_{i=0}^\infty a_i P(i),\end{equation*}
and the sum for a given entry $a_{k, n}$ is always finite.
This fact indicates that we should deduce general properties of binomial arrays from those of Pascal's triangle that are linear in nature - and perhaps some that are not.

If $p(x)$ is a polynomial of degree $m$, then the proper values of $B(p(x))$ are bordered by values with constant absolute value; the constant value is $a_0$ on the top line,  the diagonal from column zero to the lower right has constant value $a_m$, and, past row $m$,  column $-1$ alternates in sign and has constant absolute value $|\sum\limits_{i=0}^m (-1)^i a_i|.$  

%section 3
\section{Symmetries}

In this section, we consider symmetries of binomial arrays induced from triangular symmetries.  Suppose $p(x)$ is a polynomial of degree $m.$ Then the binomial array $B(p(x))$ is the join of three triangles:  a binomial trapezoid progressing in the lower left direction, a overlapping difference table, and a binomial trapezoid progressing to the right. The right-hand side of the binomial array is clear.  

To see that a binomial trapezoid occurs on the left-hand side, we may interpret the usual Pascal's recurrence, progressing to the right, as a Pascal's recurrence progressing to the lower left direction with sign corrections. With this rule, the corner of the L is the sum of the right value of the L plus the negation of the value at the top of the L.   The initial condition of this trapezoid is the diagonal of entries starting at row 0 of column $-m-1$ and ending at the entry at row $m+1$ in column $-1,$ with sign parity affected by row parity.

The left and right trapezoids are joined by a difference table with sides of length $m+2$ joining the initializing diagonal noted above to the origin.  The columns of the triangle are the given by repeatedly dividing $p^*(x)=x^mp(1/x)$  by $x+1$; each column records the quotient with remainder.  The remainders give the coefficients in the Taylor expansion of $p^*(x)$ at $x=-1.$  For example, the difference table in Figure 2 below expresses the equality
\begin{equation*}p^*(x)= 2x^3+5x^2+x-6= 2(x+1)^3-(x+1)^2-3(x+1)-4.\end{equation*}

\begin{figure}
\includegraphics[scale=.6]{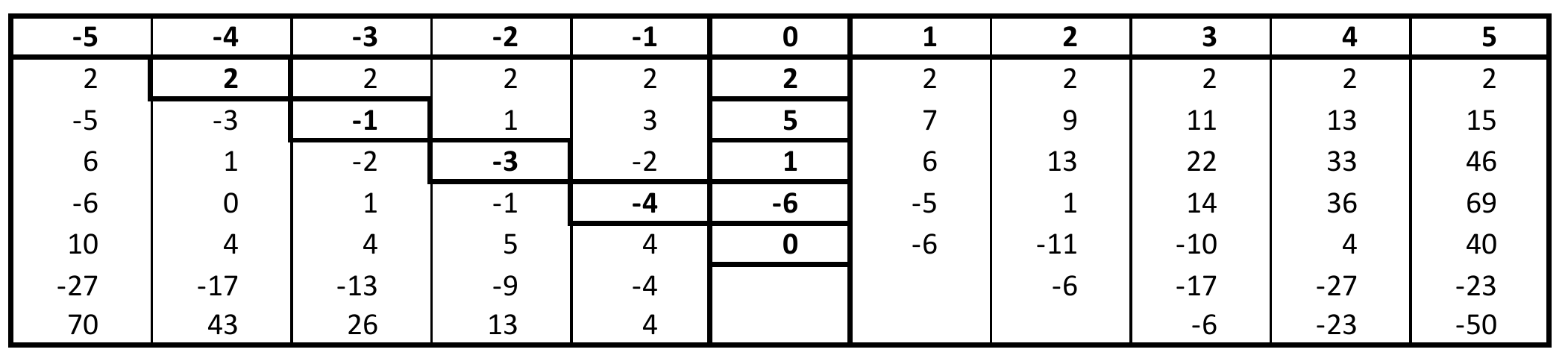}
\caption{$B(p(x))$ with $p(x)=-6x^3+x^2+5x+2$}
\label{}
\end{figure}

We have immediately
\begin{proposition}
The symmetry that interchanges the trapezoids on each side of $B(p(x))$ is induced by reflecting the difference table of $B(p(x))$ along rows, with sign changes along rows based on row parity. 
\end{proposition}

For example,
\begin{equation*}  
\begin{bmatrix*}[r]
 2  & 2   & 2   & 2 & {2} \\
  & -1 &   1  & 3   & {5}  \\
  &    & -3 & -2 & 1\\
   &    &    & -4  & -6 \\
    &    &    &   & {0}  \\
\end{bmatrix*}\quad\mapsto\quad
\begin{bmatrix*}[r]
 2  & 2   & 2   & 2 & {2} \\
  & -5 &   -3  & -1   & {1}  \\
  &    & 1 & -2 & -3 \\
   &    &    & 6  & 4  \\
    &    &    &   & {0}  \\
\end{bmatrix*}\
\end{equation*}

Next
\begin{proposition}
The involution $p(x) \mapsto p^*(x)=x^m p(1/x)$ inverts the columns of $B(p(x))$ with non-negative index.  The difference table for $B(p^*(x))$ has values on its left-hand side given by the Taylor coefficients of $p(x)$ at $x=-1$.
\end{proposition}

For example,
\begin{equation*}  
\begin{bmatrix*}[r]
 2  & 2   & 2   & 2 & {2} \\
  & -1 &   1  & 3   & {5}  \\
  &    & -3 & -2 & 1 \\
   &    &    & -4  & -6  \\
    &    &    &   & {0}  \\
\end{bmatrix*}\quad\mapsto\quad
\begin{bmatrix*}[r]
 {-6}  & -6   & -6   & -6 & -6 \\
  & 19 &   13  & 7   & 1  \\
  &    & -15 & -2 & 5\\
   &    &    & 4  & 2  \\
    &    &    &   & {0}  \\
\end{bmatrix*},\quad
\end{equation*}
One readily checks that 
\begin{equation*}p(x)=-6x^3+x^2+5x+2= -6(x+1)^3+19(x+1)^2-15(x+1)+4.\end{equation*}

Finally one reflects across the diagonal of the left trapezoid by passing to to the right side, reflecting, and passing back to the left.  To compute the right trapezoid, one inverts the left-hand side of the difference table with sign changes.

%section 4
\section{Hockey Stick Rules}

We now describe the various hockey stick rules in any $B(p(x)).$ If $p(x)$ is not a polynomial, then there are two hockey stick rules: Proposition 2.4  and its reflection under the symmetry that interchanges trapezoids. If $p(x)$
is a polynomial, we may apply the four-fold symmetries to obtain six hockey stick rules, with two additional  rules for each lower edge of the proper region. In addition, each rule has a corresponding short version, yielding three short rules after multiplicity.

\begin{proposition}[Top Line to Lower Right] 
In addition to the top line rule (Proposition 2.4), one also has
\begin{equation*}a_{0, n}+a_{1, n+1} + \dots + a_{k, n+k} = a_{k, n+k+1}.\end{equation*}
With $0<k_1<k_2$, the short rule is given by
\begin{equation*}a_{k_1-1, n+k_1}+(a_{k_1, n+k_1}+a_{k_1+1, n+k_1+1} + \dots + a_{k_2, n+k_2}) = a_{k_2, n+k_2+1}.\end{equation*}
\end{proposition}
\begin{proof}
The formula follows immediately from Proposition 2.4 after applying the symmetry that interchanges trapezoids.  We see  readily that the alternation is removed  in the case of Pascal's triangle; the case of general $B(a_i)$ follows from linearity. 
\end{proof}

The following subarrays indicate the pattern of each rule:
\begin{equation*}  
\begin{bmatrix*}[r]
 \fbox{1} & 1  & 1   & 1   & 1 & 1  \\
 6 & \fbox{7} & 8 &   9  & 10   & 11  \\
 15 & 21 &  \fbox{28}  & 36 & 45 & 55 \\
 20&  35   & 56   & \fbox{84}   & 120  & 165  \\
 29 &  49 &  84 & 140 & \fbox{224} & \fbox{344}\\
\end{bmatrix*},\quad
\begin{bmatrix*}[r]
 1 & 1  & 1   & 1   & 1 & 1  \\
 6 & 7 & \fbox{8} &   9  & 10   & 11  \\
 15 & 21 &  \fbox{28}  & 36 & 45 & 55 \\
  20 & 35   &  56   & \fbox{84}   & 120  & 165  \\
  29 &  49 &  84 & 140 & \fbox{224} & \fbox{344}\\
\end{bmatrix*}.
\end{equation*}
Of course,
\begin{equation*}224+84+36+28+7+1=344,\qquad (224+84+28)+8=344\end{equation*}

Now suppose $p(x)$ is a polynomial of degree $m$.  Using the symmetries, we obtain four additional rules, based off the diagonal on the right-hand side and the centerline of the left-hand side.  

\begin{proposition}[Right-hand Side Rules] Let $p(x)$ be a polynomial of degree $m$. Fix $n>0$, and suppose $k>m.$
\begin{enumerate}
\item Along rows from the central column of $B(p(x))$, 
\begin{equation*}a_{k, 1} + \dots + a_{k, n} = a_{k+1, n+1}.\end{equation*}
\item Along columns to the right-hand side diagonal of $B(p(x))$, 
\begin{equation*} a_{k, n}-a_{k+1, n}+\dots + (-1)^{m+n-k}a_{m+n, n} = a_{k-1, n-1}.\end{equation*}
\end{enumerate}
\end{proposition}

The following subarrays indicate the pattern of each rule with $p(x)=6x^2-6x+3$:
\begin{equation*}  
\begin{bmatrix*}[r]
 3 & 3  & 3   & 3   & 3 & 3  \\
 -6 & -3 & 0 &   3  & 6   & 9  \\
 6 & 0 &  -3  & -3 & 0 & 6 \\
 0 &  \fbox{6}   & \fbox{6}   & \fbox{3}   & 0  & 0  \\
 0  &  0 &  6 & 12 & \fbox{15} & 15\\
\end{bmatrix*},\quad
\begin{bmatrix*}[r]
 3 & 3  & 3   & 3    \\
 -6 & \fbox{$-3$} & 0 &   3    \\
 6 & 0 &  \fbox{$-3$}  & -3  \\
 0 &  6   & \fbox{6}   & 3   \\
 0  &  0 &  \fbox{6} & 12 \\
 0 & 0 & 0 & 6
\end{bmatrix*}
\end{equation*}

Of course,
\begin{equation*}6 + 6 + 3 = 15,\qquad -3+6-6=-3.\end{equation*}

\begin{proposition}[Left-hand Side Rules] Let $p(x)$ be a polynomial of degree $m$. Fix $n>0$ and suppose $k>m.$
\begin{enumerate}
\item Along rows from the centerline of $B(p(x))$,
\begin{equation*}  a_{k, -1}+\dots + a_{k, -n} = -a_{k+1, -n}.\end{equation*}
\item Along diagonals from the central column  to the upper left of $B(p(x))$, 
\begin{equation*} a_{k, -1}+a_{k-1, -2}+\dots + a_{k-n+1, -n} = -a_{k-n, -n}.\end{equation*}
\end{enumerate}
\end{proposition}

The following subarrays indicate the pattern of each rule with $p(x)=6x^2-6x+3$:
\begin{equation*}  
\begin{bmatrix*}[r]
 3 & 3  & 3   & 3   & 3 & 3  \\
 -21 & -18 & -15 &   -12  & -9   & -6  \\
 81 & 60 &  42  & 27 & 15 & 6 \\
  -225 & \fbox{-144} &  \fbox{ -84} &  \fbox{-42}  &  \fbox{-15} & 0  \\
 510 & \fbox{285}  & 141  & 57  & 15 & 0\\
\end{bmatrix*},\qquad 
\begin{bmatrix*}[r]
  3   & 3   & 3 & 3  \\
   \fbox{-15} &  -12  & -9   & -6  \\
   \fbox{42}  & 27 & 15 & 6 \\
    -84 & \fbox{-42}  & -15 & 0  \\
  141  & 57  & \fbox{15} & 0\\
    -213 &  -72 & -15 & 0\\
\end{bmatrix*}.
\end{equation*}
Of course,
\begin{equation*} -15-42-84-144=-285,\qquad 42-42+15=-(-15).\end{equation*}

In general,  the three short rules apply to any $B(a_i).$ For the remaining short rule,

\begin{proposition}[Third Short Rule]  Let $B(a_i)$ be any binomial array.   Suppose $n_1 < n_2.$ Then
\begin{equation*}a_{k+1, n_1}+(a_{k, n_1} +\dots + a_{k, n_2})=a_{k+1, n_2+1}.\end{equation*} 
\end{proposition}

The following subarray indicates the pattern of this rule:
\begin{equation*}  
\begin{bmatrix*}[r]
 1 & 1  & 1   & 1   & 1 & 1  \\
 6 & 7 & 8 &   9  & 10   & 11  \\
 15 & \fbox{21} &  \fbox{28}  & \fbox{36} & 45 & 55 \\
 20&  \fbox{35}   & 56   & 84  & \fbox{120}  & 165  \\
 29 &  49 &  84 & 140 & 224 & 344\\
\end{bmatrix*},\quad
\end{equation*}
Of course,  $35+(21+28+36)=120.$

Alternatively, the short rules may be derived by iterating Pascal's recurrence from a fixed entry in one of three directions: upwards, to the right, and to the upper left. The usual hockey stick rules result when these iterations intersect with a border.

%section 5
\section{Cauchy Product}
We now recall the background of the Cauchy algebra for Theorem 6 of \cite{EF}, recast to fit the present situation.  Here we use the binomial transform instead of finite differences; these models are equivalent by interchanging $1+x$ and $1-x$.

\begin{definition}[Convolution]  Suppose $a$ and $b$ are sequences with generating functions $p(x)$ and $q(x)$, respectively. Then the {\bf Cauchy product} (or {\bf discrete convolution} if polynomials) of $a$ and $b$ is a new sequence $a*b$ defined by 
\begin{equation}(a*b)_m=\sum\limits_{i+j=m} a_ib_j = \sum\limits_{i=0}^m a_ib_{m-i}=[x^m](p(x)q(x)),\end{equation}
where $[x^m](p(x))=a_m.$
\end{definition}
Properties of convolution follow naturally from algebraic properties of power series. For instance, we immediately note
\begin{proposition} Let $a$, $b$, and $c$  be sequences, and let $r$ be in $F.$ Then
\begin{enumerate}
\item $a*b=b*a,$
\item $(ra+b)*c=(ra)*c+b*c,$ and
\item $(a*b)*c=a*(b*c).$
\end{enumerate}
\end{proposition}

\begin{definition}  Define the sequence $e$ by ${e}_i=\delta_{i,0},$ the Kronecker delta function. 
\end{definition}

It will be convenient to use non-standard notation for binomial coefficients in this section. 

\begin{definition}  
For $k\ge 0$ and any integer $n$, define  $t^n_k =  B^ne_k=\begin{pmatrix} n \\ k\end{pmatrix}.$   
\end{definition}

It follows immediately that, for $n\ge 0$ and any integers $n_1$ and $n_2$,  
\begin{enumerate}
\item $\sum\limits_{i=0}^n t^n_ix^i = (1+x)^n$, 
\item $\sum\limits_{i=0}^\infty t^{-n}_i x^i = (1+x)^{-n},$ and
\item $t^{n_1}*t^{n_2}=t^{n_1+n_2}.$
\end{enumerate}

In the context of binomial transforms, 
\begin{equation}B^na_k = a_{k, n} = (t^n*a)_k\quad\text{and}\quad B^{-n}a_k = a_{k, -n} = (t^{-n}*a)_k.\end{equation}

We also have
\begin{enumerate}
\item $e*e=e$,
\item $e*a=a*e$ for any sequence $a$, and
\item $t^n*t^{-n}=e.$
\end{enumerate}

With the multiplication of convolution, the algebra of sequences has no zero divisors, has an identity $e$, and a sequence $a_i$ is invertible if and only if $a_0\ne 0.$

With respect to the binomial transform, we have
\begin{theorem}[Dwyer, Frankel]  Suppose $a_i$ and $b_j$ are sequences. Then, for all integers $n,$
\begin{equation}B^na*B^{-n}b = a*b.\end{equation}
\end{theorem}
\begin{proof}  Formally, 
\begin{align*}
B^na*B^{-n}b & = (t^n*a)*(t^{-n}*b)\\
&= (t^n*t^{-n})*(a*b)\\
&=e*(a*b) = a*b.
\end{align*}
Of course, if $p(x)$ and $q(x)$ represent $a$ and $b$, respectively, then
\begin{equation*}[x^m]\big( (1+x)^np(x)(1+x)^{-n}q(x)\big)=[x]^m\big(p(x)q(x)\big).\end{equation*}
\end{proof}
We recommend the interested reader to \text{MR0033265}, T. Fort's review of \cite{EF}.

In the following sections, we apply this theorem  to produce infinite families of summation formulas derived from binomial arrays.  In section 9, we express this theorem in terms of finite-dimensional representations of $SL(2, F).$

Using the binomial array construction, the following corollary describes an implementation of this theorem.  

\begin{corollary}[Generalized Vandermonde Identity]Fix  $m\ge 0,$ and let $a$ and $b$ be sequences.  For any $n$ in $\mathbb{Z},$ 
\begin{equation}\sum\limits_{i=0}^m a_ib_{m-i} = \sum\limits_{i=0}^m\ B^na_i\, B^{-n}b_{m-i}.\end{equation}
To compute using $B(a_i)$ and $B(b_i),$  
\begin{enumerate}
\item restrict $B(a_i)$ and $B(b_i)$ to the first $m+1$ rows,
\item rotate the restricted $B(b_i)$ by 180 degrees about column zero, and
\item dot products of columns in the same relative position have equal values.
\end{enumerate}
\end{corollary}

{\bf Example:} Consider the finite subarrays centered about column zero with $m=3$:
\begin{equation*}
\begin{bmatrix*}[r]
 3 & 3  & \fbox{3}   & 3   & 3  \\
 -2 & 1 & \fbox{4} &  7  & 10     \\
 0 &  -2 &  \fbox{$-1$}  & 3 & 10 \\
 0 & 0 &  \fbox{$-2$} &  -3  & 0 \\
\end{bmatrix*},\qquad 
\begin{bmatrix*}[r]
 2 & 2  & \fbox{2}   & 2   & 2  \\
 -2 & 0 & \fbox{2} &  4  & 6     \\
 1 &  -1 &  \fbox{$-1$}  & 1 & 5 \\
 -1 & 0 &  \fbox{$-1$} &  -2  & -1 \\
\end{bmatrix*}.
\end{equation*}
We rotate the second array to obtain
\begin{equation*}
\begin{bmatrix*}[r]
 3 & 3  & \fbox{3}   & 3   & 3  \\
 -2 & 1 & \fbox{4} &  7  & 10     \\
 0 &  -2 &  \fbox{$-1$}  & 3 & 10 \\
 0 & 0 &  \fbox{$-2$} &  -3  & 0 \\
\end{bmatrix*},\qquad 
\begin{bmatrix*}[r]
 -1 & -2  & \fbox{$-1$}   & 0   & -1  \\
 5 & 1 & \fbox{$-1$} &  -1  & 1     \\
 6 &  4 &  \fbox{2}  & 0 & -2 \\
 2 & 2 &  \fbox{2} &  2  & 2 \\
\end{bmatrix*}.
\end{equation*}
We see immediately that each of the five dot products between matching columns equal $-13.$

Without the rotation, the matched columns move about the initial columns in opposite directions by the same increment.  If both columns are from a single array, the matching is a telescoping effect. 
\vspace{10pt}

We note further corollaries to Theorem 5.5:
\begin{corollary}  Suppose $a^j_i$ are sequences for $1\le j \le k$. Then 
\begin{equation*}B^{n_1}a^1*B^{n_2}a^2*\dots* B^{n_k}a^k = B^{n_1+\dots +n_k}(a^1*a^2*\dots*a^k).\end{equation*}
\end{corollary}

\begin{corollary}  Suppose $(a*b)_m=0.$ Then, for $n$ in $\mathbb{Z},$
\begin{equation*}(B^{n}a*B^{-n}b)_m =0.\end{equation*}
In particular, if $a_{k, n}=0,$ then $\sum\limits_{i=0}^k \begin{pmatrix} l \\ i\end{pmatrix} a_{k-i, n-l}=0$ for all $l$ in $\mathbb{Z}$.
\end{corollary}

Finally we may apply the symmetries from section 3 to obtain further  Vandermonde identities. For general binomial arrays $B(a_i)$ and $B(b_i)$, a second convolution occurs along the top line, now along diagonals to the right.

\begin{corollary}Fix  $m\ge 0,$ and let $B(a_i)$ and $B(b_i)$ be binomial arrays.  For any $r, s, t$ in $\mathbb{Z},$ 
\begin{equation*}\sum\limits_{i=0}^m a_{i, r+i}b_{m-i,s-i} = \sum\limits_{i=0}^m\ a_{i, r+i+t}b_{m-i,s-i-t}.\end{equation*}
\end{corollary}

In the case  $B(a_i)$ and $B(b_i)$ correspond to polynomials, we obtain four additional convolution formulas along the lower edges of the proper regions, two per side.  We leave the indexing in these cases to the reader.

%section 6
\section{Examples of Sequential Zeros}
We use the following strategy for generating families of summation formulas:
\begin{enumerate}
\item choose a binomial array with an interesting sequence near a regular progression of zeros in the array, 
\item pair with a binomial array with a simple initial condition, and
\item expand a simple convolution using Corollary 5.6.
\end{enumerate}

\begin{definition}
Suppose the entries of $B(a_i)$ are denoted $a_{k, n}.$  If $w, v,$ (resp. $u, v'$) are non-negative (resp. non-zero) integers, we define a {\bf progression of arithmetic type $(v, v')$} in $B(a_i)$ to be a sequence of the form $c_t=a_{u+tv, w+tv'}$ for all $t\ge 0.$  Such a progression $c^*_t$  is called a {\bf sequence of near zeros} if $c_t=0$ always, $c_t$ and $c^*_t$ have the same type, and $max(|u-u^*|,|w-w^*|)=1.$
\end{definition}

Thus we first identify elementary rules for generating binomial arrays with progressions of zeros.  Of course, when $p(x)$ is a polynomial, there is always an infinite  triangular region of zeros in $B(p(x))$ with a side on column 0.  We note two further types of specific interest:  those with initial condition
\begin{equation*}p(x)=sx+r\end{equation*}
for some nonzero integers $r$ and $s$, and  those with a skew-palindromic initial condition.

Consider the first case with initial condition $p(x)=sx+r.$  Using scalar multiplication, this case reduces further to two cases:   $r$ is positive and $s$ is negative, or both $r$ and $s$ are positive and $r<s.$  In the first case, the progression of zeros occurs on the right-hand side of $B(p(x))$, the second on the left-hand side.  It also happens that these are the only proper zeros in $B(p(x)).$ 

\begin{proposition}  
Suppose $r$ and $s$ are non-zero integers.  If $(|r|, |s|)=d>1,$ then $B(sx+r)$ and $B(\frac{sx+r}{d})$ have the same zeros.

Otherwise suppose $r$ and $s$ are both positive with $(r, s)=1.$ Then
\begin{enumerate}
\item  $a_{k, n}$ is a proper zero of $B(-sx+r)$ if and only if  $n=l(r+s)-1$ and $k=lr$ for $l\ge 1,$ and
\item $a_{k, n}$ is a proper zero of $B(sx+r)$ if and only if  $n=l(r-s)+1$, $k=lr$, and $r<s$ for $l\le -1.$
\end{enumerate}
\end{proposition}

\begin{proof} The first statement follows from the scalar multiplication property.

Next suppose $r$ and $s$ are both positive and relatively prime. Then the general term of $B(-sx+r)$ is given by, with $n$ any integer and $0< k< n$,
\begin{equation}a_{k, n}=r\begin{pmatrix}n \\ k\end{pmatrix}-s\begin{pmatrix}n \\ k-1\end{pmatrix}.\end{equation}
When $n>0,$ this difference equals 
\begin{equation*}\frac{r(n+1)-(r+s)k}{n-k+1}\begin{pmatrix}n \\ k\end{pmatrix}.\end{equation*}
and all zeros on the right-hand side of $B(-sx+r)$ are found.  When $n<0,$ the difference becomes a sum of two terms with the same parity, so no zeros occur on the left-hand side.

The remaining case follows a similar argument.  Now the zeros occur on the left-hand side only when $r<s$, and no zeros occur on the right.
\end{proof}

Examples of each case are given in Figures 5 and 6, respectively. In the general case, $B(-sx+r)$ and $B((r+s)x+r)$ interchange trapezoids.  Furthermore, we may consider the right-hand side of $B(-sx+r)$ to be a direct limit of hexagons of Clebsch-Gordan coefficients (\cite{Do}, \cite{D2}) under inclusion:
\begin{equation}M(r, s, 1)\hookrightarrow \frac{1}{2}M(2r, 2s, 1) \hookrightarrow \frac13 M(3r, 3s, 1) \hookrightarrow \frac14 M(4r, 4s, 1)\hookrightarrow \dots.\end{equation}

We now identify some values adjacent to these zeros. 
\begin{definition}[Figure 3]  Fix positive integers $r$ and $s$.  Define the integer sequence
\begin{equation*}C_t^{(r, s)}=\frac{1}{t}\frac{(rt+st)!}{(rt+1)!(st-1)!}=\frac{1}{t}\begin{pmatrix} rt+st\\ rt+1\end{pmatrix}\end{equation*}
for $t\ge 1.$
\end{definition}
Alternatively, one has
\begin{equation*}C_t^{(r, s)}=\frac{r+s}{rt+1}\begin{pmatrix} rt+st-1\\ rt\end{pmatrix}\end{equation*}
This definition gives another generalization of the Catalan numbers $C_t$, which occur when $r=s=1.$  When $s=1$, we obtain the Fu{\ss}-Catalan numbers.

\begin{figure}
\includegraphics[scale=.65]{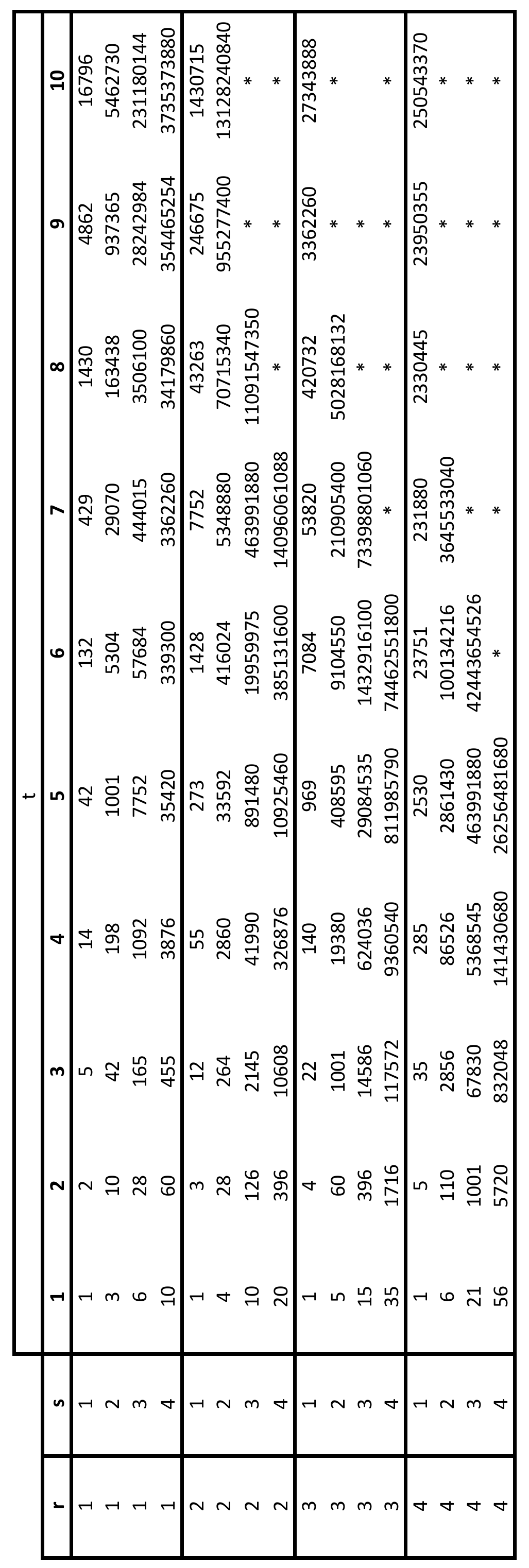}
\caption{Near-zero sequences $C_t^{(r, s)}$ for small parameters}
\label{}
\end{figure}

Up to sign, the sequence $C_t^{(r, s)}$ follows the zeros of $B(-sx+r)$ from below and to the lower right, while the sequence $C_t^{(s, r)}$ follows from above and to the right.  When $r=s$, $C_t^{(r, r)}=rC_{rt}$. This case is somewhat degenerate, as $rC_t$ follows the zeros of $B(-rx+r)=rB(-x+1).$

\vspace{10pt}

Next we define the palindromic conditions:
\begin{definition}
Let $p(x)$ be a polynomial of degree $m$.  $p(x)$ is  called {\bf palindromic} if $p(x)=x^mp(1/x)$.  On the other hand, if $p(x)=-x^mp(1/x),$ we call $p(x)$ {\bf skew-palindromic}. 
That is, $p(x)$ is palindromic (resp. skew-palindromic) if, 
\begin{equation*}\text{for}\quad 0\le k\le m,\quad a_{k}=a_{m-k}\ \text{(resp.}\ a_k=-a_{m-k}\text{)}.\end{equation*}
\end{definition}

If $p(x)$ is palindromic or skew-palindromic, then so is $(1+x)^np(x)$ for all $n\ge 0$; that is, a palindromic (resp. skew-palindromic) initial condition produces a binomial array with palindromic (resp. skew-palindromic) columns on the right. If $m$ is even and $p(x)$ is skew-palindomic, then the middle entry equals zero, and the following proposition holds immediately.
\begin{proposition} Suppose $p(x)$ is skew-palindromic of degree $m$.
\begin{enumerate}
\item If $m=2l,$ then $a_{l}=0.$ There is a diagonal of zeros in $B(p(x))$ to the right in even-numbered columns; specifically, for all $k$ in $\mathbb{N},$  $a_{l+k, 2k}=0.$
\item If $m=2l+1,$ then $a_{l+1, 1}=0.$  There is a diagonal of zeros in $B(p(x))$ to the right in odd-numbered columns; specfically, for all $k$ in $\mathbb{N}$, $a_{l+k+1, 2k+1}=0.$
\end{enumerate}
\end{proposition}
The simplest non-trivial example of the skew-palindromic case is given by the Catalan triangle array in Figure 5.  Vertical symmetry guarantees that four values about a zero on the diagonal have the same absolute value (below, above, to the right, to the lower right).  Using the symmetry that switches the initial condition to the trapezoid on the left, we obtain an unbroken diagonal of zeros on the left.  See Figure 6.

 \begin{figure}
 \includegraphics[scale=.6]{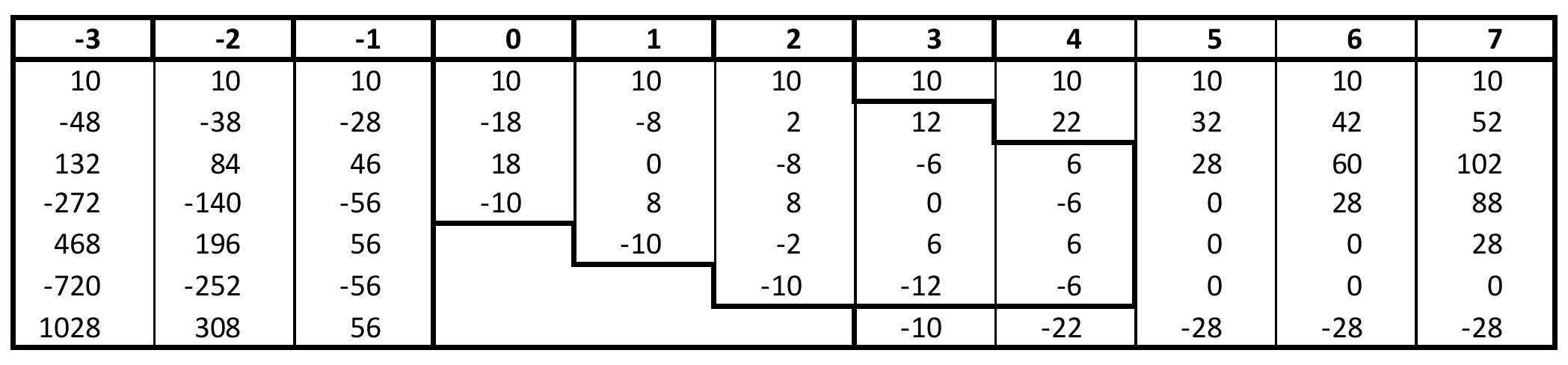}
\caption{$M(5, 5, 3)$ as a subarray followed by triangle of zeros}
\label{}
\end{figure}

We note two general sources of skew-palindromic initial conditions:
\begin{enumerate}
\item polynomials of the form $p_{r, s}(x)=(1-x^r)^s$ for positive integers $r$ and $s$ with odd $s$, and 
\item hexagons of Clebsch-Gordan coefficients (\cite{Do}, \cite{D2}) of the form $M(m, m, k)$ with $k$ odd and $0< k < m$.
\end{enumerate}
The first case is just the aeration of the coefficients of $(1-x)^s$ with strings of $r-1$ zeros.  Then $p_{r, s}$ is palindromic (resp. skew-palindromic) if and only if $s$ is even (resp. odd). 

For the second case, the initial condition is given by 
\begin{equation*}a_i = (-1)^i \begin{pmatrix}m-i\\ k-i\end{pmatrix}\begin{pmatrix}m-k+i\\ i\end{pmatrix}\end{equation*}
for $0\le i \le k,$ and $a_i=0$ otherwise.  Since the hexagon $M(r, r, 1)$ is a subarray of $B(-rx+r),$ we obtain yet another generalization of the Catalan numbers. 

Finally we note another source of zeros from hexagons of Clebsch-Gordan coefficients.  If the hexagon $M(m, n, k)$ is subsection of the binomial array $B(a_i)$, then there is a right triangle of zeros of height $k$ to the right of the hexagon. See Figure 4.  For a survey of vanishing rules for Clebsch-Gordan coefficients, see \cite{D2}.

%section 7
\section{Example: Pascal's Triangle}

To set up a general framework, we review the classical case of Pascal's triangle.  Recall that the extended Pascal's triangle is $B(e)$ with entries $e_{k, n}.$ For fixed $n$ and  $k>0,$ we may section off the rectangle from the origin to $e_{k, n}$ to obtain three cases:
\begin{equation*}  
\begin{bmatrix}
  e_{0, -n} & \dots &  \fbox{$1$}      \\
  e_{1, -n} &  \dots &   0  \\
 \dots   & \dots  & \dots \\
  \fbox{$e_{k, -n} $}  & \dots & 0  \\
 \end{bmatrix},\qquad 
 \begin{bmatrix}
    \fbox{$1$}      \\
     0  \\
 \dots   \\
  \fbox{$e_{k, 0} $}   \\
 \end{bmatrix},\qquad
 \begin{bmatrix}
 \fbox{$1$}      & \dots & e_{0, n}  \\
 0 &  \dots &    e_{1, n}   \\
 \dots   & \dots  & \dots \\
  0    & \dots   & \fbox{$e_{k, n} $} \\
 \end{bmatrix}.
\end{equation*}
Applying Corollary 5.6, we obtain
\begin{equation*}\sum\limits_{i=0}^k e_{i, 0-l} e_{k-i, n+l}=e_{k, n},\end{equation*}
which summarizes as 
 \begin{proposition} [Chu-Vandermonde] Fix integers $m$ and $n$ and positive integer $k$.    Then
\begin{equation*}\sum\limits_{i=0}^k \begin{pmatrix} m\\ i\end{pmatrix} \begin{pmatrix} n \\ k-i\end{pmatrix}= \begin{pmatrix} m+n \\ k \end{pmatrix}.\end{equation*}
\end{proposition}
Of course, we obtain four non-vanishing cases based on sign parity, and a vanishing sum occurs when $0 < m+n < k.$ Of course, all cases may be obtained directly and simply using a product of binomial series expansions.  In terms of Section 5, this is just item (3) before (5.2):
\begin{equation*}(B^me*B^ne)_k=B^{m+n}e_k.\end{equation*}

Now suppose we choose $l$ in Corollary 5.6 such that the columns translate to either the same or adjacent columns and sum over the length of the right-hand column.  Then we obtain the classical special formulas:
\begin{equation}\sum\limits_{i=0}^n \begin{pmatrix} n\\ i\end{pmatrix}^2= \begin{pmatrix} 2n \\ n \end{pmatrix},\qquad  \sum\limits_{i=0}^{n} \begin{pmatrix} n\\ i\end{pmatrix}\begin{pmatrix} n+1\\ i\end{pmatrix}= \begin{pmatrix} 2n+1 \\ n \end{pmatrix}.\end{equation}
Two principles here serve as a model for the general case.  We note the first as another corollary to Theorem 5.5. 
\begin{corollary} For sequences $a_i$ and $b_i$ and integer $n$, we have
\begin{equation*}B^na*B^nb=a*B^{2n}b\quad \text{and}\quad B^na*B^{n+1}b=a*B^{2n+1}b.\end{equation*}
\end{corollary}

 In general, if $p(x)$  has degree $m$, then the convolution of a column of $B(p(x))$ with itself is given by
\begin{equation*}\sum\limits_{i=0}^{n+m} a_{i, n}a_{n+m-i, n}=\sum\limits_{i=0}^m \begin{pmatrix}2n\\ n+m-2i \end{pmatrix}a_i^2 + 2\sum\limits_{ i<j} \begin{pmatrix}2n\\ n+m-i-j \end{pmatrix}a_ia_j, \end{equation*}
and the convolution of adjacent columns of $B(p(x))$ is given by
\begin{align*}
\sum\limits_{i=0}^{n+m+1} a_{i, n}a_{n+m+1-i, n+1} &=\sum\limits_{i=0}^m \begin{pmatrix}2n+1\\ n+m+1-2i \end{pmatrix}a_i^2\\ 
&\qquad\qquad\qquad\qquad + 2\sum\limits_{ i<j} \begin{pmatrix}2n+1\\ n+m+1-i-j \end{pmatrix}a_ia_j.\end{align*}
These equalities follow immediately when expressed as
\begin{equation*}[x^{n+m}]((1+x)^{2n}(p(x))^2)\quad\text{and}\quad [x^{n+m+1}]((1+x)^{2n+1}(p(x))^2).\end{equation*}

We obtain simple closed formulas when $m=1$:

\begin{theorem}
Fix non-zero integers $r$ and $s$.  For $n>0,$ the columns of $B(sx+r)$ satisfy
\begin{equation*}\sum\limits_{i=0}^{n+1} a_{i, n}a_{n+1-i, n}=\big[(r+s)^2n+2rs\big] C_n\end{equation*}
and
\begin{equation*}\sum\limits_{i=0}^{n+2} a_{i, n}a_{n+2-i, n+1}=\frac{1}2\big[(r+s)^2n+4rs+2s^2\big] C_{n+1}.\end{equation*}
\end{theorem}
Of course, choosing odd $r$ and even $s$ in the second equality implies, for positive $n$, $C_n$ odd implies $n$ odd. In fact (\cite{AC}; \cite{Ko}, Theorem 13.1), $C_n$ is odd if and only if $n=2^k-1.$
\vspace{10pt}

Next the palindromic or skew-palindromic conditions allow further simplification of our sums.

\begin{proposition}
Suppose $p(x)=a_0+\dots+a_mx^m$ and $q(x)=b_0+\dots+b_mx^m$ have degree $m,$ and denote the entries of $B(p(x))$ $($resp. $B(q(x)))$ by $a_{k, n}$ $($resp. $b_{k, n})$.  If $p(x)$ and $q(x)$ are palindromic, then, for all $l$ in $\mathbb{Z}$,
\begin{align*} \sum_{k=0}^m a_{k, -l}b_{m-k, l}= \sum_{k=0}^{m} a_kb_{m-k}  &=2\sum_{k=0}^{\frac{m-1}2} a_kb_k\quad (m\ \text{odd}), \\
&=a_{\frac{m}2}b_{\frac{m}2}+2\sum_{k=0}^{\frac{m-2}2} a_kb_k\quad (m\ \text{even}).\end{align*}
On the other hand, if $p(x)$  and $q(x)$ are both skew-palindromic, then
\begin{equation*} \sum_{k=0}^m a_{k, -l}b_{m-k, l}= \sum_{k=0}^m a_{k}b_{m-k}=-2\sum_{k=0}^{\lfloor \frac{m}2\rfloor} a_kb_k. \end{equation*}
Finally, if $p(x)$  is palindromic and $q(x)$ is skew-palindromic, then
\begin{equation*} \sum_{k=0}^m a_{k, -l}b_{m-k, l}= \sum_{k=0}^m a_{k}b_{m-k}=0.\end{equation*}
\end{proposition}

\begin{proposition}
Suppose $p(x)=a_0+\dots+a_mx^m$ has degree $m$ and $q(x)=b_0+\dots+b_{m+1}x^{m+1}$ has degree $m+1$.  If  $p(x)$ is palindromic, then, for all $l$ in $\mathbb{Z}$,
\begin{equation*} \sum_{k=0}^m a_{k, -l}b_{m-k+1, l+1}= \sum_{k=0}^{m} a_kb_{m-k+1}  =\sum_{k=0}^m a_kb_{k+1}.\end{equation*}
If $p(x)$ is palindromic and $b_i=Ba_i,$ then
\begin{equation*}(B^{-l}a*B^{l+1}a)_{m+1}=(a*a)_m+(a*a)_{m-1}.\end{equation*}
On the other hand, if $p(x)$ is skew-palindromic, then
\begin{equation*} \sum_{k=0}^m a_{k, -l}b_{m-k+1, l+1}= \sum_{k=0}^m a_{k}b_{m-k+1}=-\sum_{k=0}^m a_kb_{k+1}. \end{equation*}
If $p(x)$ is skew-palindromic and $b_i=Ba_i$, then
\begin{equation*}(B^{-l}a*B^{l+1}a)_{m+1}=-(a*a)_{m}+(a*a)_{m-1}.\end{equation*}
\end{proposition}

%section 8
\section{Example: Catalan Numbers}

In this section, we consider various binomial arrays concerning the Catalan numbers, both directly and indirectly.  Following the extended Pascal's triangle, the next simplest initial conditions are $1\pm x$.  Now $1+x$ is the second column of the extended Pascal's triangle, so we consider the array $B(1-x)$ and its reflection $B(1+2x).$  For this section, we reserve $a_{k, n}$ (resp. $b_{k, n}$) for the values in the former (resp. latter).

\begin{figure}
\includegraphics[scale=.6]{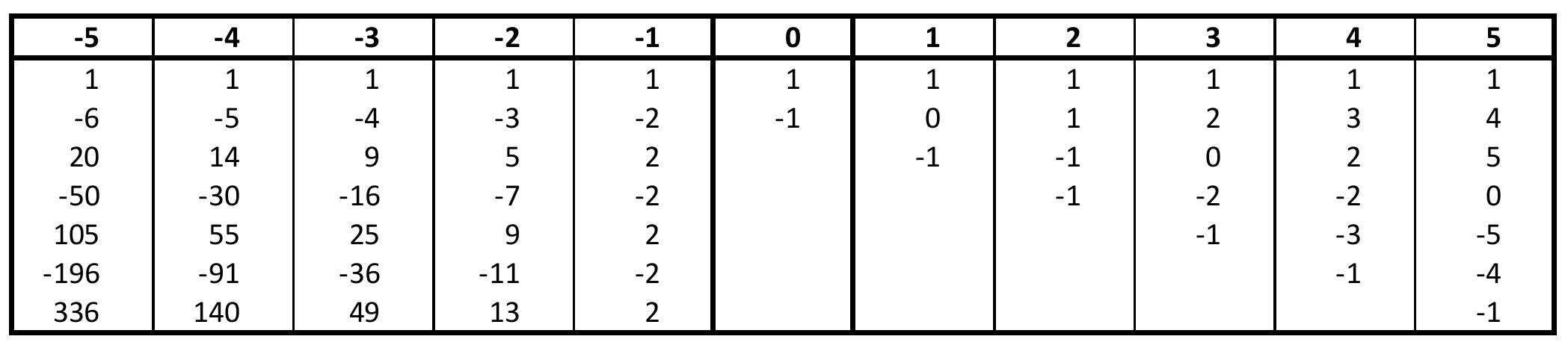}
\caption{Catalan number array $B(1-x)$: $C_{\frac{n+1}2}$ on the right, below/above zeros}
\label{}
\end{figure}

One has immediately that 
\begin{proposition}[Figure 5]  
For integers $k$ and $n$, 
\begin{equation}a_{k, n}  = \begin{pmatrix} n\\ k \end{pmatrix}- \begin{pmatrix} n\\ k-1 \end{pmatrix} \end{equation}
Specifically, for $n> 0$ and $0<k\le n,$ 
\begin{equation}a_{k, n} = \frac{n-2k+1}{n-k+1} \begin{pmatrix} n\\ k \end{pmatrix}\end{equation}
and, for $n>1$ and $k>0,$
\begin{equation}a_{k, -1}=(-1)^{k+1}2,\quad a_{k, -n} = (-1)^k\ \frac{-n-2k+1}{n-1} \begin{pmatrix} n+k-2\\ k \end{pmatrix}\end{equation}
\end{proposition}

By the skew-palindromic property or by direct substitution, we obtain a progression of zeros of type $(1, 2)$ starting at $a_{1, 1}$. If we consider the near zero sequence at points $(t, 2t)$, we have 
\begin{equation*}a_{t, 2t}=\frac{1}{t+1} \begin{pmatrix} 2t\\ t \end{pmatrix}=C_t,\end{equation*}
the sequence of Catalan numbers.  Alternatively, we note
\begin{definition}[Segner's Recurrence]
The Catalan numbers are defined recursively as follows:
\begin{equation}C_0=1,\qquad C_{n+1}=\sum\limits_{i=0}^n C_iC_{n-i}.\end{equation}
\end{definition}
See the first rows of Figures 3 or 8 for $C_1$ through $C_{10},$ or see \cite{Ko}, Appendix B, for the first 100 Catalan numbers. Interpretations of the Catalan numbers by way of enumeration are vast in number;  see  \cite{St2} for over 200 examples.   

In Segner's recurrence, we see that the Catalan numbers are defined by iterated convolution. For comparison with Definition 6.3, we also note the alternative formula for $t\ge 1,$
\begin{equation}C_t=\frac{1}{t}\frac{(2t)!}{(t+1)!(t-1)!}=\frac{1}{t}\begin{pmatrix} 2t\\ t+1 \end{pmatrix}.\end{equation}

In its right-hand side trapezoid, $B(1-x)$ is subject to all rules governing Clebsch-Gordan coefficients from \cite{Do} and \cite{D2}; that is, we may consider this trapezoid as a direct limit of hexagons of Clebsch-Gordan coefficients under inclusion:
\begin{equation}\frac{1}{r}  M(r, r, 1) \hookrightarrow \frac{1}{r+1}M(r+1, r+1, 1) \hookrightarrow \frac{1}{r+2}M(r+2, r+2, 1)\hookrightarrow \dots.\end{equation}

If we now consider $B(1+2x)$,  one has 
\begin{proposition}[Figure 6] For positive integers $n$ and $k$,
\begin{equation*}b_{k, n}= \begin{pmatrix} n\\ k \end{pmatrix}+2 \begin{pmatrix} n\\ k-1 \end{pmatrix} = \frac{n+k+1}{n-k+1} \begin{pmatrix} n\\ k \end{pmatrix},\end{equation*}
and
\begin{align*}b_{k, -n}&= (-1)^{k}\bigg[\begin{pmatrix} n+k-1 \\ k \end{pmatrix}-2 \begin{pmatrix} n+k-2\\ k-1 \end{pmatrix}\bigg]\\
& = (-1)^{k}\ \frac{n-k-1}{n+k-1}\begin{pmatrix} n+k-1\\ k \end{pmatrix}\end{align*}
\end{proposition}

Since the initial condition is not skew-palindromic, we use Proposition 8.3 to see that the only proper zeros in this array are at points $(t, -t-1)$ for $t\ge 1.$ To the left of these zeros, we obtain
\begin{equation*}B_{t, -t-2}=(-1)^t \frac{1}{2t+1}\begin{pmatrix}2t+1\\ t\end{pmatrix}=(-1)^t C_{t}\end{equation*}

Now we apply Theorem 7.3 to extend Shapiro's formulas for Catalan triangle numbers in Corollary 8.4 and 8.5.  The skew-palindromic property with truncation yields the dot product of rows in Shapiro's Catalan triangle.

\begin{corollary}  For $n\ge 0,$ the sum of squares the values of the  $n$-th column of $B(1-x)$ is $2C_n$. That is, 
\begin{equation} C_n =\sum\limits_{i=0}^{\lfloor\frac{n}2\rfloor} a_{i, n}^2,\end{equation}
and, for all $l$ in $\mathbb{Z},$
\begin{equation} C_n= -\frac{1}{2}\sum\limits_{i=0}^{n+1} a_{i, n-l}a_{n+1-i, n+l}.
\end{equation}
\end{corollary}

\begin{corollary}  For $n\ge 0,$ the convolution of the  $n$-th and $(n+1)$-st columns of $B(1-x)$ is $-C_{n+1}$. That is, 
\begin{equation} C_{n+1} =-\sum\limits_{i=0}^{n+1} a_{i, n}a_{n+1-i, n+1}\end{equation}
and, for all $l$ in $\mathbb{Z},$
\begin{equation} 
C_{n+1} =-\sum\limits_{i=0}^{n+1} a_{i, n-l}a_{n+1-i, n+l+1}.
\end{equation}
\end{corollary}

\begin{figure}
\includegraphics[scale=.6]{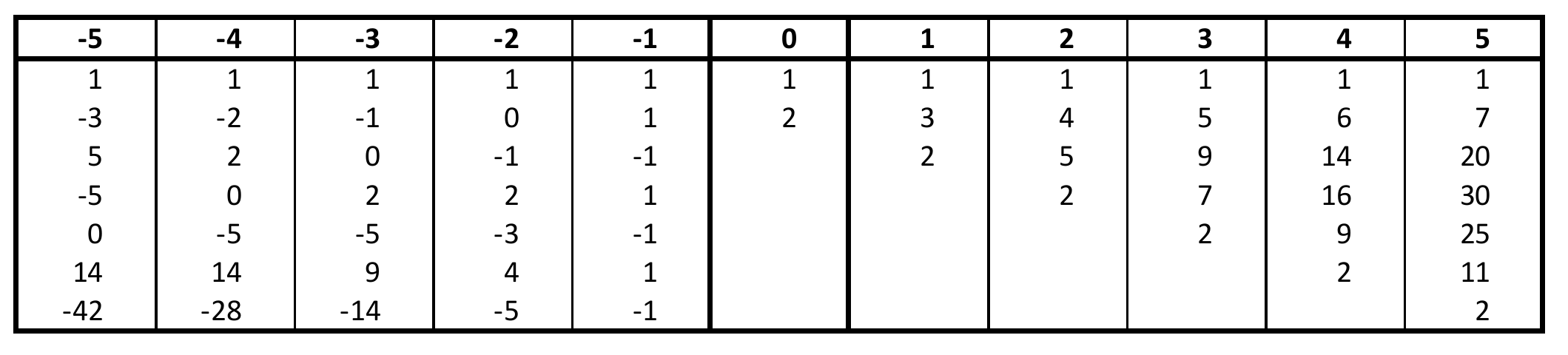}
\caption{Catalan number array $B(1+2x)$:  $C_{-n-1}$ on the left, below zeros}
\label{}
\end{figure}

Likewise, for $B(1+2x)$,

\begin{corollary}  For $n\ge 0,$ the sum of squares the values of the  $n$-th column of $B(1+2x)$ is $(9n+4)C_n$. That is, 
\begin{equation} (9n+4)C_{n} =\sum\limits_{i=0}^{n+1} b_{i, n}^2,\end{equation}
and, for all $l$ in $\mathbb{Z},$
\begin{equation} (9n+4)C_{n}= \sum\limits_{i=0}^{n+1} b_{i, n-l}b_{n+1-i, n+l}.
\end{equation}
\end{corollary}

\begin{corollary}  For $n\ge 0,$ the convolution of the  $n$-th and $(n+1)$-st columns of $B(1+2x)$ is $\frac12 (9n+16)C_{n+1}$. That is, 
\begin{equation} (9n+16)C_{n+1} =2\sum\limits_{i=0}^{n+1} b_{i, n}b_{n+1-i, n+1}\end{equation}
and, for all $l$ in $\mathbb{Z},$
\begin{equation} 
(9n+16)C_{n+1} =2\sum\limits_{i=0}^{n+1} b_{i, n-l}b_{n+1-i, n+l+1}.
\end{equation}
\end{corollary}

On the other hand, we may directly match the initial condition to Catalan numbers adjacent to zeros in three ways, as noted below:
\begin{align*}
C_{2n-1} &= -\sum\limits_{i=0}^{2n} b_{i, -n} b_{2n-i, -n} = -\frac12\sum\limits_{i=0}^{2n} b_{i, -n-1} b_{2n-i, -n}\\ &= \sum\limits_{i=0}^{2n-1} b_{i, -n-1} b_{2n-i-1, -n}.\\
\end{align*}
\begin{align*}
C_{2n} &= \sum\limits_{i=0}^{2n+1} b_{i, -n} b_{2n-i+1, -n-1} = \frac12\sum\limits_{i=0}^{2n+1} b_{i, -n-1} b_{2n-i+1, -n-1}\\ &= -\sum\limits_{i=0}^{2n} b_{i, -n-1} b_{2n-i, -n-1}.\\
\end{align*}
Now we consider the near zero sequence of the binomial array associated to $M(m, m, k)$ with $k$ odd and $0<k<m$.  

\begin{theorem}[Figure 7] Suppose $1<k<m$ with $k=2k'+1.$  Define $C_t(m, k)$ to be the entry to the right of  zero in column $2t-1$ on the diagonal of the extended Clebsch-Gordan hexagon $M(m, m, k)$.  For $t\ge 1,$
\begin{equation*}C_t(m, k)=(m-k')\begin{pmatrix} m-k'-1\\ k'\end{pmatrix}\begin{pmatrix} t+2k'-m\\ k'\end{pmatrix}\begin{pmatrix} t+k'+1\\ k'\end{pmatrix}^{-1}C_t,\end{equation*}
where $C_t$ is the $t$-th Catalan number. When $k=1,$ $C_t(m, 1)=mC_t.$
\end{theorem}

\begin{figure}
\includegraphics[scale=.75]{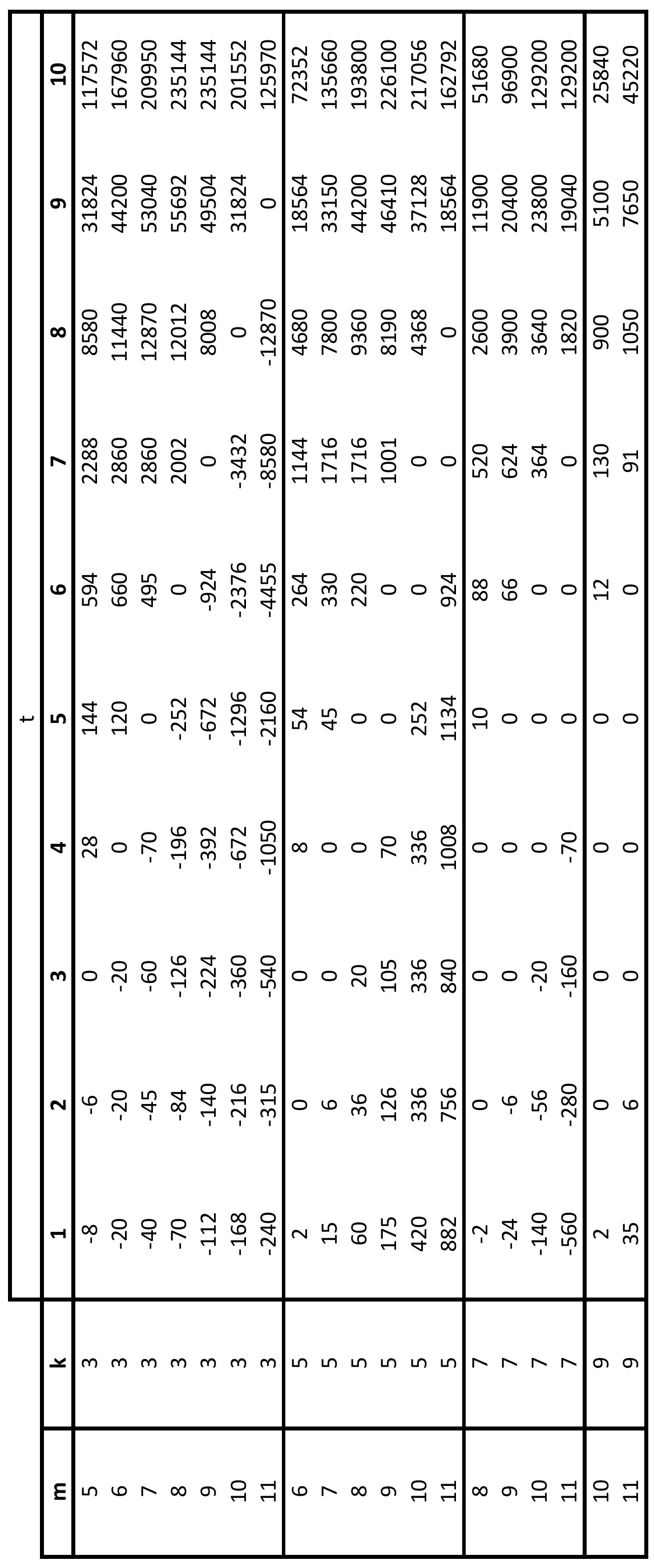}
\caption{Near-zero sequences $C_t(m, k)$ with $1<k<m$ and $k$ odd}
\label{}
\end{figure}

\begin{proof}
For an elementary proof, we follow the methods of Proposition 4.2 of \cite{D2}, but two potential issues require attention: use of the reverse recurrence inside $M(m, m, k)$, and failure of recursion if we encounter the triangle of zeros on the right-hand side of the $M(m, m, k).$  To circumvent these issues, we proceed as follows.  

First fix odd $k>1;$  if $k=1,$ then $M(m, m, 1)$ is a subarray of $B(m-mx).$  The initial condition is a column of length $k+1,$ so the shape of the array is fixed as $m$ varies.  Writing binomial coefficients as Pochhammer symbols, we may consider the entries of the initial condition as polynomials in $m$ of degree $k.$ Since Pascal's recurrence is linear, all proper values of the array may be considered as polynomials in $m$ of degree less than or equal to $k.$ Since the proper entries of $M(m, m, k)$ are nested by shape in $B(a_i)$ as $m$ increases, the polynomial for a fixed entry in some $M(m, m, k)$ is fixed in the entire array.  That is, we may restrict our calculations within some $M(m, m, k).$

Now the proof follows by calculating the  recurrence relation and the initial condition.  By the position of zeros and the skew-palindromic property, we have the following subarray relating $C_t(m, k)$ to $C_{t+1}(m, k):$

\begin{equation*}  
\begin{bmatrix*}[r]
  C_t  & * & * & *  \\
  0 & \fbox{$C_t$}  & \fbox{$C_{t+1}$}   &  * \\
  -C_t & -C_t &  \fbox{$0$} & C_{t+1}    \\
 * & * & -C_{t+1} & -C_{t+1} \\
 \end{bmatrix*}.
\end{equation*}
\vspace{10pt}

Using the reverse recurrence on the boxed entries with $i=k'+t$ and $j=k'+t+1,$ we obtain
\begin{equation}C_{t+1}(m, k)=\frac{2(2t+1) (m-t-k)}{(t+k'+2)(m-t-k'-1)} C_{t}(m, k).\end{equation}

For the initial condition $C_1$, we sum the initial values $a_{k'-1}$ and $a_{k'}$ 
\begin{equation*}(-1)^{k'-1}\begin{pmatrix}m-k'+1\\ k'+2\end{pmatrix} \begin{pmatrix}m-k'-2\\ k'-1\end{pmatrix}+(-1)^{k'}\begin{pmatrix}m-k'\\ k'+1\end{pmatrix}\begin{pmatrix}m-k'-1\\ k'\end{pmatrix}\end{equation*}
to get
\begin{equation}C_1(m, k)=(-1)^{k'}\frac{2}{k'+2}\begin{pmatrix}m-k'-2\\ k'\end{pmatrix} \begin{pmatrix}m-k'\\ k'+1\end{pmatrix}.\end{equation}
\end{proof}

Note that  $C_t(m, k)$ is a polynomial in $m$ of degree less than or equal to $k$, and that $C_t(m, k)=0$ for $m-2k'\le t \le m-k'-1$.  Otherwise, the $t$-th zero on the diagonal of $M(m, m, k)$ occurs at coordinates $(k'+t, 2t-1)$ in $B(a_i).$ So $C_t(m, k)$ is the value at coordinates $(k'+t, 2t),$ which in the notation of \cite{D2} is given by 
\begin{align}c_{m, m, k}(k'+t, k'+t+1)&=\sum\limits_{l=0}^k (-1)^l\begin{pmatrix} 2t \\ k'+t-l\end{pmatrix} \begin{pmatrix}m-l\\ k-l\end{pmatrix}\begin{pmatrix} m-k+l\\ l\end{pmatrix}.\end{align}

Now we consider the near zero sequence of a general skew-palindromic binomial array $B(d_i)$ with degree $m=2m'+1$;  if $m$ is even, we may replace the initial condition $d_i$ with $B^1d_i.$ We decompose $B(d_i)$ into a sum of skew-palindromic binomial arrays of the form $B(1-x^{2i+1}).$
\begin{lemma}[Figure 8]
Fix $r\ge 0.$ The diagonal of entries of $B(-x+1)$ at coordinates $(l, r+2l)$ are given by the sequence 
\begin{equation*}a_{l, r+2l}=\frac{r+1}{r+l+1}\begin{pmatrix} r+2l\\ l\end{pmatrix}\end{equation*}
\end{lemma}
\begin{proof}
This follows immediately from substitution into the formula for $a_{k, n}$ above.
\end{proof}

\noindent {\it Remark.}  The $r$-th sequence is the $(r+1)$-th convolution power of $C_t.$

\begin{proposition}  (a)  For $t\ge 1,$ let $c_{2i+1}(t)$ be the value to the right of the zero in column $2t-1$ on the diagonal of $B(1-x^{2i+1}).$  
Then $c_1(t)=C_t$, and, for $i\ge 1,$ $c_{2i+1}(t)$ is given by the values of Lemma 8.9 with $r=2i$ and $l=t-i$.  That is, for $i \ge 1,$
\begin{equation}c_{2i+1}(t)=\frac{2i+1}{t+i+1}\begin{pmatrix}2t\\ t-i\end{pmatrix}.\end{equation}
This sequence is  the $2i$-th row  of Figure 8, preceded by $i-1$ zeros.

(b)   For $t\ge 0,$ let $c_{2i}(t)$ be the value to the right of the zero in column $2t$ on the diagonal of $B(1-x^{2i}).$  
Then $c_2(t)=C_{t+1}$, and, for $i\ge 1,$ $c_{2i}(t)$ is given by the values of Lemma 8.9 with $r=2i-1$ and $l=t-i+1$.  That is, for $i \ge 1,$
\begin{equation}c_{2i}(t)=\frac{2i}{t+i+1}\begin{pmatrix}2t+1\\ t-i+1\end{pmatrix}.\end{equation}
This sequence is  the $(2i-1)$-th row  of Figure 8, preceded by $i-1$ zeros.
\end{proposition}
\begin{proof}
(a) Suppose the entries of $B(1-x^{2i+1})$  are denoted by $c_{k, n}$ and $i\ge 1.$ Then the entry to the right of the $t$-th zero has coordinates $(i+t, 2t)$. Recall that $B(x^{2i+1})=P(2i+1)$, the extended Pascal's triangle shifted down by $2i+1$ rows.  Then 
\begin{equation*}c_{k, n} = \begin{pmatrix} n\\ k \end{pmatrix}-\begin{pmatrix}n \\ k-2i-1 \end{pmatrix}.\end{equation*}
Thus 
\begin{equation*}c_{i+t, 2t} = \begin{pmatrix} 2t\\ i+t \end{pmatrix}-\begin{pmatrix} 2t \\ t-i-1 \end{pmatrix}
= \frac{(2t)!(2i+1)}{(t+i+1)!(t-i)!},\end{equation*}
and the lemma holds.

A similar proof holds for part (b).
\end{proof}

\begin{figure}
\includegraphics[scale=.65]{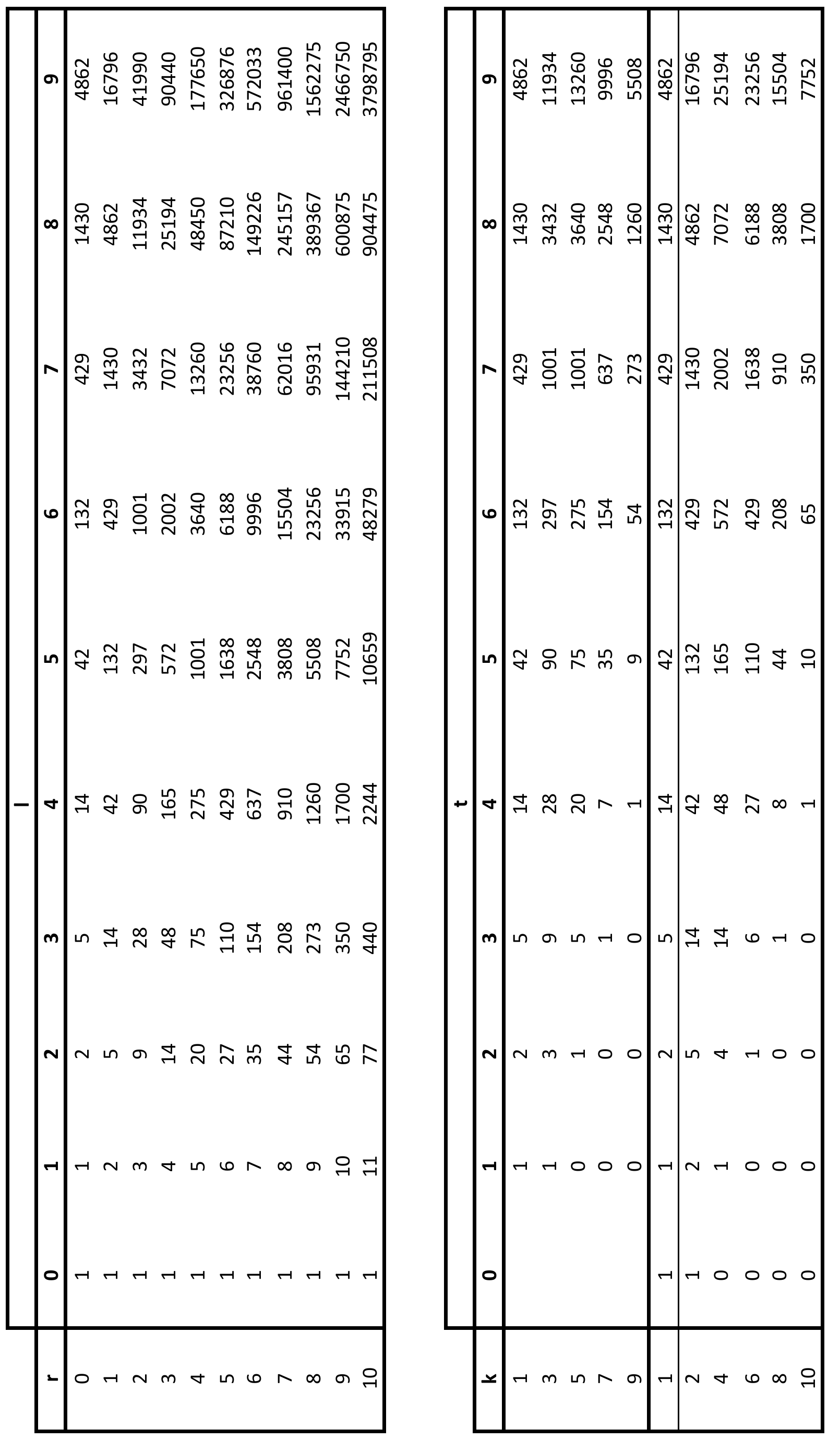}
\caption{$(1, 2)$-progressions $a_{l, r+2l}$ in the array $B(1-x)$(left);
sequences $c_k(t)$(right)}
\label{}
\end{figure}

 By linearity and the skew-palindromic property of  the first column of $B(d_i)$, we may express the corresponding binomial array as
\begin{equation*}\sum\limits_{i=0}^{m'} d_i B(x^{i}(1-x^{m-2i})).\end{equation*} 
The diagonal zeros occur at coordinates $(m'+t, 2t-1)$, so the $t$-th value of interest has coordinates $(m'+t, 2t)$.
Combining the initial condition with Proposition 8.10, we obtain
\begin{theorem}  Suppose  $p(x)$ is skew-palindromic of degree $m=2m'+1$ with corresponding sequence $d_i$, and $c_k(t)$ is defined as in Proposition 8.10. 
Then, for $t\ge 1,$  the near zero sequence $D_t$ of the binomial array $B(d_i)$ satisfies
\begin{equation}D_t=d_{m'+t, 2t}=\sum\limits_{i=0}^{m'}\ d_i\, c_{m-2i}(t),
\end{equation} 
\end{theorem}

To implement this formula as a dot product, the left-hand side of Figure 8 is separated into rows by parity to give the table on the right-hand side.  Note that  the columns of the even rows are  the rows of Shapiro's Catalan triangle.

\begin{corollary} (a) For $t\ge 1$ and $m=2m'+1,$
\begin{equation}c_{m+1}(t+1)=\sum\limits_{i=0}^{m'}\ (-1)^i\, c_{m-2i}(t).
\end{equation} 
(b) For $t\ge 1$ and $m=2m'$,
\begin{equation}c_{m+1}(t)=(-1)^{m'}C_t+\sum\limits_{i=0}^{m'-1}\ (-1)^i\, c_{m-2i}(t).
\end{equation} 
\end{corollary}
\begin{proof}
(a)  Suppose $d_i=(-1)^i$ for $0 \le i\le m$ and zero otherwise.  If $p(x)$ corresponds to $d_i$, then $(1+x)p(x)=1-x^{m+1},$ and the result follows.

(b)  Suppose $p(x)=\sum\limits_{i=0}^{m'-1} (-1)^i x^i(1-x^{m-2i}).$ Then $$(1+x)p(x)=1-(-1)^{m'}(x^{m'} -x^{m'+1})-x^{m+1},$$ and the result follows.
\end{proof}

Of course, this corollary may be realized as a special case of the vertical short hockey stick rule (Proposition 2.5) for $B(1-x).$  Additionally, part (a) is the case of $C_t(m, m)$ when $m$ is odd;  compare with Theorem 8.8.

%section 9
\section{Realization through representation theory of  $SL(2, F)$}

In this section, we recast Corollary 5.6 in terms of the finite-dimensional representation theory of $SL(2, F).$  Recall that $F$ is a field of characteristic zero; in general, we only consider irreducible representations of $SL(2, F)$ with a highest weight $n$.  When $F=\mathbb{R}$ or $\mathbb{C}$, see, for instance, the earlier chapters of \cite{Kn} for definitions and a convenient model (Chapter 2).

Define $M(2, F)$ to be the set of square matrices of size 2 with coefficients in $F$.  Let $GL(2, F)$ be the group of invertible matrices in $M(2, F)$ under matrix multiplication, and let $G=SL(2, F)$ be the subgroup of matrices with determinant 1.  Let $\mathfrak{g}_0=\mathfrak{sl}(2, F)$ be the Lie algebra of traceless matrices in $M(2, F)$ under the bracket $[X, Y]=XY-YX,$ and let $U(\mathfrak{g}_0)$ be the universal enveloping algebra of $\mathfrak{g}_0.$  We express the model of previous sections in terms of the representation theory of $SL(2, F)$ and $\mathfrak{sl}(2, F).$

We choose the usual basis for $\mathfrak{g}_0$ with
\begin{equation}e=\begin{bmatrix} 0 & 1\\ 0 & 0 \end{bmatrix},\quad  f=\begin{bmatrix} 0 & 0\\ 1 & 0 \end{bmatrix}, \quad h=\begin{bmatrix*}[r] 1 & 0\\ 0 & -1 \end{bmatrix*}.\end{equation}

Fix $n\ge 0,$ and let $(\pi_n, V(n))$ be the irreducible representation of $G$ with highest weight $n.$ When extending $\pi_n$ to elements of $U(\fg_0),$ we typically omit $\pi_n$ from the notation.
If we choose a highest weight vector $\phi$, then $V(n)$ has dimension $n+1$ with basis given by
\begin{equation}\mathcal{B}_n=\{ \phi, f\phi, \dots,  f^n\phi\}.\end{equation}
Up to scalar, $V(n)$ admits a unique non-degenerate invariant bilinear form $\langle\cdot, \cdot\rangle_n,$ and we fix this form by the normalization
\begin{equation*}\langle f^n\phi, \phi\rangle_n = 1. \end{equation*}
If $X$ is in $\fg_0,$  invariance implies that 
\begin{equation*}\langle Xu, v\rangle_n=\langle u, -Xv\rangle_n.\end{equation*}
Thus we have for all $0\le k, l\le n$,
\begin{equation}\langle f^{l} \phi, f^k\phi\rangle_n = \begin{cases}(-1)^k & l=n-k, \\ 0 & \text{otherwise}. \end{cases}\end{equation}
Now this form is symmetric when $n$ is even, alternating when $n$ is odd, and the dual basis to $\mathcal{B}_n$ with respect to $\langle\cdot, \cdot\rangle_n$ is given by
\begin{equation}\mathcal{B}'_n=  \{(-f)^n\phi,  \dots, (-f)\phi,  \phi\}.\end{equation}

Next let 
\begin{equation*}s=\begin{bmatrix*}[r]  1 & 0\\ 0 & -1\end{bmatrix*} \end{equation*}
and consider the involution $S: SL(2, F)\rightarrow SL(2, F)$  defined by
\begin{equation*} Sg=sgs^{-1}\quad\text{or}\quad S\begin{bmatrix}a & b\\ c & d\end{bmatrix} = \begin{bmatrix*}[r]a & -b\\ -c & d\end{bmatrix*}.\end{equation*}
The induced involution on $\mathfrak{sl}(2, F)$ is given by extending
\begin{equation*}e\mapsto -e, \quad f\mapsto -f, \quad h\mapsto h\end{equation*}
This involution induces an equivalent representation $(\pi'_n, V(n))$ by 
$$\pi'_n(g)v=\pi_n(Sg)v$$
with induced involution $S$ given on $V(n)$ by extending $S\phi=\phi$ and $f^k\phi\mapsto (-f)^k\phi.$
That is, we have an equivalence by $S\pi_n(g)=\pi'_n(g)S.$

Consider the form on $V(n)$ defined by 
\begin{equation*}\langle u, v\rangle'_n=\langle u, Sv\rangle_n.\end{equation*}
Now we have for all $0\le k, l\le n$,
\begin{equation}\langle f^{l} \phi, f^k\phi\rangle'_n = \begin{cases}1 & l=n-k, \\ 0 & \text{otherwise}, \end{cases}\end{equation}
and we shall see that this form implements the convolution pairing on polynomials.  It is no longer invariant under the action by $\pi_n$ but rather
\begin{align*}
\langle\pi_n(g)u, \pi'_n(g)v \rangle'_n &= \langle \pi_n(g)u, S\pi'_n(g)v\rangle_n\\
&=\langle \pi_n(g)u, \pi_n(g)Sv\rangle_n\\
&=\langle u, Sv\rangle_n\\
& = \langle u, v\rangle'_n.
\end{align*}

For $v$ in $V(n)$, define 
\begin{equation}B_f: V(n)\rightarrow V(n)\qquad  \text{by} \qquad B_fv=(1+f)v.\end{equation}
Since $\pi_n(f)$ is nilpotent, $B_f$ is invertible on $V(n),$ and 
\begin{equation}B_f^{-1}v = \frac{1}{1+f} v = (1 + (-f) + \dots + (-f)^n)v.\end{equation}
From above,
\begin{equation*}
\langle \pi_n(f)u, v \rangle'_n  = \langle u, \pi_n(f)v\rangle'_n.
\end{equation*}
Thus $B_f$ is symmetric with respect to $\langle\cdot, \cdot\rangle'_n$, and it follows immediately that
\begin{equation*}\langle B_fu, B^{-1}_fv\rangle'_n = \langle u, v\rangle'_n.\end{equation*}
If $p(X)$ and $q(X)$ are polynomials with coefficients in $F$ of degree less than or equal to $n$, we define
\begin{equation*}\langle\langle p(X), q(X)\rangle\rangle_n = \langle p(f)\phi, q(f)\phi\rangle'_n\end{equation*}
and reacquire
\begin{equation}\langle\langle B^1p(X), B^{-1}q(X)\rangle\rangle_n = \langle\langle p(X), q(X)\rangle\rangle_n = [X^n](p(X)q(X)).\end{equation}

%    Bibliographies can be prepared with BibTeX using amsplain,
%    amsalpha, or (for "historical" overviews) natbib style.
\bibliographystyle{amsplain}

\end{document}